\documentclass[11pt]{article}
\usepackage{authblk}
\usepackage{amssymb}
\usepackage[all]{xy}
\usepackage{mathrsfs}
\usepackage{amsmath,amssymb,amsthm}
\usepackage{extarrows}
\usepackage[mathscr]{eucal}
\usepackage{mathrsfs}
\newtheorem{theorem}{Theorem}[]

\newtheorem{definition}[theorem]{Definition}

\newtheorem{remark}[theorem]{Remark}

\newtheorem{conjecture}[theorem]{Conjecture}
\def\f{\mathfrak}

\def\c{\mathcal}

\date{}
\begin{document}
%%%%%%%%%%%%%%%%%%%%%%%%%%%%%%%%%%%%%%%%%%%%%%%%%%%%%%%%%%%%%%%%%%%%%%%%%%%%%%%%%%%%%%%%%%%%%%%%%%%%%%%%%%%%%%%%%%%%%%%%%%
\title{A Note on the Frankl Conjecture}
%%%%%%%%%%%%%%%%%%%%%%%%%%%%%%%%%%%%%%%%%%%%%%%%%%%%%%%%%%%%%%%%%%%%%%%%%%%%%%%%%%%%%%%%%%%%%%%%%%%%%%%%%%%%%%%%%%%%%%%%%%
\author{Maysam Maysami Sadr\thanks{Email: sadr@iasbs.ac.ir}}
\affil{Department of Mathematics, Institute for Advanced Studies in Basic Sciences, Zanjan, Iran}
%%%%%%%%%%%%%%%%%%%%%%%%%%%%%%%%%%%%%%%%%%%%%%%%%%%%%%%%%%%%%%%%%%%%%%%%%%%%%%%%%%%%%%%%%%%%%%%%%%%%%%%%%%%%%%%%%%%%%%%%%%
\maketitle
%%%%%%%%%%%%%%%%%%%%%%%%%%%%%%%%%%%%%%%%%%%%%%%%%%%%%%%%%%%%%%%%%%%%%%%%%%%%%%%%%%%%%%%%%%%%%%%%%%%%%%%%%%%%%%%%%%%%%%%%%%
\begin{abstract}
The Frankl conjecture (called also union-closed sets conjecture) is one of the famous unsolved
conjectures in combinatorics of finite sets. It has been formulated since mid-1970s as follows:
\textit{If $\c{A}$ is a union-closed family of subsets of a finite
set $X$, which contains a nonempty subset of $X$,
then there exists $x\in X$ such that $x$ belongs to at least half of the members of $\c{A}$.}
In this short note, we introduce and to some extent
justify some variants of the Frankl conjecture.



\end{abstract}
%%%%%%%%%%%%%%%%%%%%%%%%%%%%%%%%%%%%%%%%%%%%%%%%%%%%%%%%%%%%%%%%%%%%%%%%%%%%%%%%%%%%%%%%%%%%%%%%%%%%%%%%%%%%%%%%%%%%%%%%%%
%%%%%%%%%%%%%%%%%%%%%%%%%%%%%%%%%%%%%%%%%%%%%%%%%%%%%%%%%%%%%%%%%%%%%%%%%%%%%%%%%%%%%%%%%%%%%%%%%%%%%%%%%%%%%%%%%%%%%%%%%%
%%%%%%%%%%%%%%%%%%%%%%%%%%%%%%%%%%%%%%%%%%%%%%%%%%%%%%%%%%%%%%%%%%%%%%%%%%%%%%%%%%%%%%%%%%%%%%%%%%%%%%%%%%%%%%%%%%%%%%%%%%
%%%%%%%%%%%%%%%%%%%%%%%%%%%%%%%%%%%%%%%%%%%%%%%%%%%%%%%%%%%%%%%%%%%%%%%%%%%%%%%%%%%%%%%%%%%%%%%%%%%%%%%%%%%%%%%%%%%%%%%%%%
Let $X$ be a finite nonempty set and let $\c{P}(X)$ denote the powerset of $X$.
By a union-closed family on $X$ we mean a set $\c{A}\subseteq\c{P}(X)$
such that $\emptyset,X\in\c{A}$, and such that if $A_1,A_2\in\c{A}$ then $A_1\cup A_2\in\c{A}$.
Frankl's conjecture (FC for short), that is also called union-closed sets conjecture,
is as follows.
\begin{conjecture}
\emph{FC:} If $\c{A}$ is a union-closed family on a finite nonempty set $X$,
then there exists $x\in X$ such that $x$ belongs to at least half of the members of $\c{A}$.
\end{conjecture}
For history of FC, some basic related results, and some equivalent formulations of the conjecture in
terms of graphs and lattices, we refer the reader to
\cite{BruhnSchaudt1,Poonen1}
and references therein. In this short note, we introduce some natural variants of FC.

Let us begin by some notations and definitions.
For a set $S$, we denote by $|S|$ the cardinal of $S$. If $\c{A}$ is a union-closed family on $X$, we denote $X$ by
the symbol $\cup\c{A}$. The class of all union-closed families on finite nonempty sets is denoted by $\f{A}$.
\begin{definition}
Let $k$ and $\ell$ be two natural numbers with $k\geq\ell\geq1$.
A family $\c{A}\in\f{A}$ is called $k|\ell$-separated if $|\cup\c{A}|\geq k$ and for every
$k$ distinct elements $x_1,\ldots,x_k\in\cup\c{A}$, there exists $A\in\c{A}$ such that $x_1,\ldots,x_\ell\in A$ and
$\{x_{\ell+1},\ldots,x_k\}\cap A=\emptyset$. The class of all $k|\ell$-separated union-closed families is denoted by $\f{A}_{k|\ell}$.
\end{definition}
By definition, any union-closed family is $1|1$-separated; thus $\f{A}_{1|1}=\f{A}$.
Also, $\f{A}_{k|k}$ is the class of all union-closed families $\c{A}$ with $|\cup\c{A}|\geq k$.
As a trivial example, for every finite set $X$ with $|X|\geq k$, any union-closed
family on $X$, containing the union-closed family $\{A\subseteq X:|A|\geq\ell\}\cup\{\emptyset\}$,
is $k|\ell$-separated.

For every two natural numbers $k$ and $\ell$ as above,
let $\f{c}_{k|\ell}$ be the supremum of all positive real numbers $\epsilon$ satisfying the following condition:
For every $\c{A}\in\f{A}_{k|\ell}$, there exists a set $S\subseteq\cup\c{A}$ with $|S|=k$ such that
there are at least $\epsilon|\c{A}|$ members $A$ of $\c{A}$ satisfying $|S\cap A|\geq\ell$.
If there is no such an $\epsilon$, then $\f{c}_{k|\ell}$ is defined to be $0$. Note that the `supremum' in the definition
of $\f{c}_{k|\ell}$ is actually a `maximum'.
\begin{theorem}\label{1808290122}
Let $k,k',k'',\ell$ be natural numbers such that $k,k'\geq\ell\geq1$ and $k''\geq1$. The following four inequalities are satisfied.
\begin{enumerate}
\item[(1)] $0\leq \f{c}_{k|\ell}\leq2^{-k}\sum_{i=\ell}^k\binom{k}{i}$
\item[(2)] $\f{c}_{k+k'|\ell}\geq \f{c}_{k|\ell}+\f{c}_{k'|\ell}-\f{c}_{k|\ell}\f{c}_{k'|\ell}$
\item[(3)] $\f{c}_{k+k''|\ell+k''}\geq\f{c}_{k|\ell}\f{c}_{k''|k''}$
\item[(4)] $\f{c}_{k+k'|k+k'}\geq\f{c}_{k|k}\f{c}_{k'|k'}$
\end{enumerate}
\end{theorem}
\begin{proof}
(1) Let $X$ be a finite set with $|X|\geq k$. Then, $\c{P}(X)\in\f{A}_{k|\ell}$. It is easily checked that
for every subset $S\subseteq X$ with $|S|=k$, the number of those members $A$ of $\c{P}(X)$ with
$|A\cap S|\geq\ell$, is equal to $2^{|X|-k}\sum_{i=\ell}^k\binom{k}{i}$. This shows that (1) is satisfied.

(2) Let $\c{A}\in\f{A}_{k+k'|\ell}$. Then, it is easily checked that $\c{A}$ also belongs to $\f{A}_{k|\ell}$.
Thus, there exists $S\subset\cup\c{A}$ with $|S|=k$, such that if the family $\c{B}$ is defined by
$$\c{B}:=\{A\in\c{A}:|A\cap S|\geq\ell\},$$
then $|\c{B}|\geq\f{c}_{k|\ell}|\c{A}|$. Let the real number $r\geq0$ be such that $|\c{B}|=\f{c}_{k|\ell}|\c{A}|+r$.
Suppose that the family $\c{B}'$ is defined to be the empty family if $\ell=1$, and otherwise,
$$\c{B}':=\{A\in\c{A}:\ell>|A\cap S|\geq1\}.$$
Also, let
$$\c{A}':=\{A\in\c{A}:A\cap S=\emptyset\}.$$
It is clear that $\c{A}$ coincides with the disjoint union $\c{B}\cup\c{B}'\cup\c{A}'$. Thus,
$$|\c{A}|=|\c{B}|+|\c{B}'|+|\c{A}'|.$$
On the other hand, it is easily checked that $\c{A}'$ is a union-closed family on the
set $\cup\c{A}\setminus S$ and $\c{A}'\in\f{A}_{k'|\ell}$. Thus, there exists $S'\subseteq\cup\c{A}\setminus S$ with $|S'|=k'$
such that for at least $\f{c}_{k'|\ell}|\c{A}'|$ members $A'$ of $\c{A}'$, $|A'\cap S'|\geq\ell$.
Let $p$ be the number of those members $A$ of $\c{A}$ with the property $|A\cap (S\cup S')|\geq\ell$. Then,
\begin{equation*}
\begin{split}
p&\geq|\c{B}|+\f{c}_{k'|\ell}|\c{A}'|(1+|\c{B}'|)\\
&=|\c{B}|+\f{c}_{k'|\ell}(|\c{A}|-|\c{B}|-|\c{B}'|)(1+|\c{B}'|)\\
&=|\c{B}|+\f{c}_{k'|\ell}|\c{A}|-\f{c}_{k'|\ell}|\c{B}|+\f{c}_{k'|\ell}(|\c{A}|-|\c{B}|-|\c{B}'|-1)\\
&\geq|\c{B}|+\f{c}_{k'|\ell}|\c{A}|-\f{c}_{k'|\ell}|\c{B}|\\
&=\f{c}_{k|\ell}|\c{A}|+r+\f{c}_{k'|\ell}|\c{A}|-\f{c}_{k'|\ell}\f{c}_{k|\ell}|\c{A}|-r\f{c}_{k'|\ell}\\
&=(\f{c}_{k|\ell}+\f{c}_{k'|\ell}-\f{c}_{k|\ell}\f{c}_{k'|\ell})|\c{A}|+(1-\f{c}_{k'|\ell})r\\
&\geq(\f{c}_{k|\ell}+\f{c}_{k'|\ell}-\f{c}_{k|\ell}\f{c}_{k'|\ell})|\c{A}|
\end{split}
\end{equation*}
The above inequality together with $|S\cup S'|=k+k'$ show that (2) is satisfied.

(3) Let $\c{A}\in\f{A}_{k+k''|\ell+k''}$. Then, $\c{A}\in\f{A}_{k''|k''}$, and there exists a set $S\subset\cup\c{A}$ with
$|S|=k''$ such that $S$ is contained in at least $\f{c}_{k''|k''}|\c{A}|$ members of $\c{A}$. Let
$$\c{A}'':=\{A\setminus S:A\in\c{A}\hspace{2mm}\text{and}\hspace{2mm}S\subseteq A\}\cup\{\emptyset\}.$$
Therefore, we have $|\c{A}''|\geq\f{c}_{k''|k''}|\c{A}|$.
On the other hand, it is easily seen that $\c{A}''\in\f{A}_{k|\ell}$. Thus, there exists
a set $S''\subseteq\cup\c{A}''$ with $|S''|=k$ such that there are at least $\f{c}_{k|\ell}$ members $A''$ of $\c{A}''$
satisfying $|A''\cap S''|\geq\ell$. Therefore, $S\cup S''$ has $k+k''$ elements and there exists at least
$\f{c}_{k|\ell}\f{c}_{k''|k''}|\c{A}|$ members $A$ of $\c{A}$ satisfying $|A\cap(S\cup S'')|\geq \ell+k''$.
This implies that (3) is satisfied.

(4) follows directly from (3).
\end{proof}
Now, we introduce our variant of the Frankl conjecture.
\begin{conjecture}\label{1808311916}
\emph{FC of order $k|\ell$:} If $\c{A}\in\f{A}_{k|\ell}$ then there exists a subset $S\subseteq\cup\c{A}$ with
$|S|=k$ such that the number of those members $A$ of $\c{A}$ with $|A\cap S|\geq\ell$, is at least
$2^{-k}\sum_{i=\ell}^k\binom{k}{i}$. In other words, we have
$\f{c}_{k|\ell}=2^{-k}\sum_{i=\ell}^k\binom{k}{i}$.
In particular, $\f{c}_{k|1}=2^{-k}(2^k-1)$, $\f{c}_{k|k}=2^{-k}$, and $\f{c}_{2k+1|k+1}=2^{-1}$.
\end{conjecture}
It is clear that FC of order $1|1$, is the original Frankl conjecture.
We have suggested Conjecture \ref{1808311916}, from inequality (1) in Theorem \ref{1808290122}, and special cases of the
conjecture of orders $k|1$ and $k|k$.
\begin{theorem}\label{1809012149}
Let $k,k'$ be arbitrary natural numbers.
The following statements are satisfied.
\begin{enumerate}
\item[(1)]
Validity of FC of orders $k|1$ and $k'|1$ imply validity of FC of order $k+k'|1$.
In particular, validity of the Frankl conjecture implies validity of FC of order $k|1$.
\item[(2)] Validity of FC of orders $k|k$ and $k'|k'$ imply validity of FC of order $k+k'|k+k'$.
In particular, validity of the Frankl conjecture implies validity of FC of order $k|k$.
\end{enumerate}
\end{theorem}
\begin{proof}
(1) follows from inequalities (1) and (2) of Theorem \ref{1808290122}. (2) follows from inequalities
(1) and (4) of Theorem \ref{1808290122}.
\end{proof}
It is easily seen that the following two inequalities are satisfied:
$$\Bigg[2^{-(k+k')}\sum_{i=\ell}^{k+k'}\binom{k+k'}{i}\Bigg]\geq
\Bigg[2^{-k}\sum_{j=\ell}^{k}\binom{k}{j}\Bigg]+\Bigg[2^{-k'}\sum_{j'=\ell}^{k'}\binom{k'}{j'}\Bigg]-
\Bigg[2^{-k}\sum_{j=\ell}^{k}\binom{k}{j}\Bigg]\Bigg[2^{-k'}\sum_{j'=\ell}^{k'}\binom{k'}{j'}\Bigg]$$
$$\Bigg[2^{-(k''+k)}\sum_{i=\ell}^{k''+k}\binom{k''+k}{i}\Bigg]\geq\Bigg[2^{-k''}\Bigg]\Bigg[2^{-k}\sum_{i=\ell}^{k}\binom{k}{i}\Bigg]$$
Hence, Conjecture \ref{1808311916} is compatible with inequalities (2) and (3) of Theorem \ref{1808290122}.
This is another justification for Conjecture \ref{1808311916}.

We suggest that FC of order $k|\ell$ is also satisfied in a bigger class of union-closed families than $\f{A}_{k|\ell}$:
Let $\c{A}$ be a union-closed family such that $|\cup\c{A}|\geq k$. We say that $\c{A}$ is weakly $k|\ell$-separated if
the following condition is satisfied: For any $k$ distinct elements $x_1,\ldots,x_k\in\cup\c{A}$,
there exists $A\in\c{A}$ such that $\{x_1,\ldots,x_\ell\}\cap A\neq\emptyset$ and
$\{x_{\ell+1},\ldots,x_k\}\cap A=\emptyset$. we denote the class of all weakly $k|\ell$-separated
union-closed families by $\widetilde{\f{A}}_{k|\ell}$. It is clear that $\f{A}_{k|\ell}\subset\widetilde{\f{A}}_{k|\ell}$.
We let the real number $\widetilde{\f{c}}_{k|\ell}\geq0$ be exactly defined as $\f{c}_{k|\ell}$ except that this time
the supremum being taken over $\widetilde{\f{A}}_{k|\ell}$. Obviously, we have $\widetilde{\f{c}}_{k|\ell}\leq\f{c}_{k|\ell}$.
\begin{conjecture}\label{1809012210}
\emph{Strong FC of order $k|\ell$:} $\widetilde{\f{c}}_{k|\ell}=2^{-k}\sum_{i=\ell}^k\binom{k}{i}$.
\end{conjecture}
Note that the statements and proofs of Theorems \ref{1808290122} and \ref{1809012149} hold if the symbols $\f{c},\f{A}$
are replaced by $\widetilde{\f{c}},\widetilde{\f{A}}$, and the phrase `FC of order' by `strong FC of order'.
Since $\widetilde{\f{A}}_{k|1}=\f{A}_{k|1}$ (resp. $\widetilde{\f{A}}_{k|k}=\f{A}_{k|k}$), strong FC and FC of
order $k|1$ (resp. $k|k$) coincide.
\begin{remark}
It is remarked that
we have checked the validity of Conjectures \ref{1808311916} and \ref{1809012210} for families $\c{A}$ with $|\cup\c{A}|\leq5$.
\end{remark}
%%%%%%%%%%%%%%%%%%%%%%%%%%%%%%%%%%%%%%%%%%%%%%%%%%%%%%%%%%%%%%%%%%%%%%%%%%%%%%%%%%%%%%%%%%%%%%%%%%%%%%%%%%%%%%%%%%%%%%%%%%
%%%%%%%%%%%%%%%%%%%%%%%%%%%%%%%%%%%%%%%%%%%%%%%%%%%%%%%%%%%%%%%%%%%%%%%%%%%%%%%%%%%%%%%%%%%%%%%%%%%%%%%%%%%%%%%%%%%%%%%%%%
%%%%%%%%%%%%%%%%%%%%%%%%%%%%%%%%%%%%%%%%%%%%%%%%%%%%%%%%%%%%%%%%%%%%%%%%%%%%%%%%%%%%%%%%%%%%%%%%%%%%%%%%%%%%%%%%%%%%%%%%%%
\bibliographystyle{amsplain}

\begin{thebibliography}{10}
%%%%%%%%%%%%%%%%%%%%%%%%%%%%%%%%%%%%%%%%%%%%%%%%%%%%%%%%%%%%%%%%%%%%%%%%%%%%%%%%%%%%%%%%%%%%%%%%%%%%%%%%%%%%%%%%%%%%%%%%%%
\bibitem{BruhnSchaudt1}
H. Bruhn, O. Schaudt,
\textit{The journey of the union-closed sets conjecture},
Graphs Combin. 31, no. 6 (2015): 2043--2074. (arXiv:1309.3297 [math.CO])
%%%%%%%%%%%%%%%%%%%%%%%%%%%%%%%%%%%%%%%%%%%%%%%%%%%%%%%%%%%%%%%%%%%%%%%%%%%%%%%%%%%%%%%%%%%%%%%%%%%%%%%%%%%%%%%%%%%%%%%%%%
\bibitem{Poonen1}
B. Poonen,
\textit{Union-closed families},
Journal of Combinatorial Theory, Series A 59, no. 2 (1992): 253--268.
%%%%%%%%%%%%%%%%%%%%%%%%%%%%%%%%%%%%%%%%%%%%%%%%%%%%%%%%%%%%%%%%%%%%%%%%%%%%%%%%%%%%%%%%%%%%%%%%%%%%%%%%%%%%%%%%%%%%%%%%%%
%%%%%%%%%%%%%%%%%%%%%%%%%%%%%%%%%%%%%%%%%%%%%%%%%%%%%%%%%%%%%%%%%%%%%%%%%%%%%%%%%%%%%%%%%%%%%%%%%%%%%%%%%%%%%%%%%%%%%%%%%%
\end{thebibliography}

%%%%%%%%%%%%%%%%%%%%%%%%%%%%%%%%%%%%%%%%%%%%%%%%%%%%%%%%%%%%%%%%%%%%%%%%%%%%%%%%%%%%%%%%%%%%%%%%%%%%%%%%%%%%%%%%%%%%%%%%%%
%%%%%%%%%%%%%%%%%%%%%%%%%%%%%%%%%%%%%%%%%%%%%%%%%%%%%%%%%%%%%%%%%%%%%%%%%%%%%%%%%%%%%%%%%%%%%%%%%%%%%%%%%%%%%%%%%%%%%%%%%%
%%%%%%%%%%%%%%%%%%%%%%%%%%%%%%%%%%%%%%%%%%%%%%%%%%%%%%%%%%%%%%%%%%%%%%%%%%%%%%%%%%%%%%%%%%%%%%%%%%%%%%%%%%%%%%%%%%%%%%%%%%
%%%%%%%%%%%%%%%%%%%%%%%%%%%%%%%%%%%%%%%%%%%%%%%%%%%%%%%%%%%%%%%%%%%%%%%%%%%%%%%%%%%%%%%%%%%%%%%%%%%%%%%%%%%%%%%%%%%%%%%%%%
%%%%%%%%%%%%%%%%%%%%%%%%%%%%%%%%%%%%%%%%%%%%%%%%%%%%%%%%%%%%%%%%%%%%%%%%%%%%%%%%%%%%%%%%%%%%%%%%%%%%%%%%%%%%%%%%%%%%%%%%%%
\end{document}